\documentclass[11pt]{article}
\usepackage[english]{babel}
\usepackage{amssymb,amsmath,amsthm}
\newtheorem{theorem}{Theorem}[section]
\newtheorem{lemma}[theorem]{Lemma}
\newtheorem{proposition}[theorem]{Proposition}

{\theoremstyle{definition}}
{\theoremstyle{definition}}

\numberwithin{equation}{section}

\def\C{{\mathbb C}}
\def\N{{\mathbb N}}
\def\Z{{\mathbb Z}}
\def\R{{\mathbb R}}

\def\T{{\mathbb T}}

\def\hh{{\mathcal H}}

\def\epsilon{\varepsilon}
\def\phi{\varphi}
\def\leq{\leqslant}
\def\geq{\geqslant}

\title{A complete locally convex space of countable
dimension admitting an operator with no invariant subspaces}

\author{Stanislav Shkarin}

\date{}

\begin{document}

\maketitle

\begin{abstract}We construct a complete locally convex topological
vector space $X$ of countable algebraic dimension and a continuous
linear operator $T:X\to X$ such that $T$ has no non-trivial closed
invariant subspaces.
\end{abstract}


\section{Introduction \label{s1}}\rm

All vector spaces in this article are over the field $\C$ of complex
numbers. As usual, $\R$ is the field of real numbers,
$\T=\{z\in\C:|z|=1\}$, $\N$ is the set of positive integers and
$\Z_+=\N\cup\{0\}$. Throughout the article, all topological spaces
{\it are assumed to be Hausdorff}. For a topological vector space
$X$, $L(X)$ is the algebra of continuous linear operators on $X$ and
$X'$ is the space of continuous linear functionals on $X$. For $T\in
L(X)$, the dual operator $T':X'\to X'$ is defined as usual:
$T'f=f\circ T$.

We say that a topological vector space $X$ has the {\it invariant
subspace property} if every $T\in L(X)$ has a non-trivial
(=different from $\{0\}$ and $X$) closed invariant subspace. The
problem  whether $\ell_2$ has the invariant subspace property is
known as the invariant subspace problem and remains perhaps the
greatest open problem in operator theory. It is worth noting that
Read \cite{read1} and Enflo \cite{enf} (see also \cite{isp-book})
showed independently that there are separable infinite dimensional
Banach spaces, which do not have the invariant subspace property. In
fact, Read \cite{read2} demonstrated that $\ell_1$ does not have the
invariant subspace property. All existing constructions of operators
on Banach spaces with no invariant subspaces are rather
sophisticated and artificial. On the other hand, examples of
separable non-complete normed spaces with or without the invariant
subspace property are easy to construct.

\begin{proposition}\label{prop1} Every normed space of countable
algebraic dimension does not have the invariant subspace property.
On the other hand, in every separable infinite dimensional Banach
space $B$, there is a dense linear subspace $X$ such that $X$ has
the invariant subspace property.
\end{proposition}

Thus completeness is an essential difficulty in constructing
operators with no non-trivial invariant subspaces. Countable
algebraic dimension is often perceived as almost incompatible with
completeness. Basically, there is only one complete topological
vector space of countable dimension, most analysts are aware of.
Namely, the locally convex direct sum $\phi$ (see \cite{rob}) of
countably many copies of the one-dimensional space $\C$ has
countable dimension and is complete. In other words, $\phi$ is a
vector space of countable dimension endowed with the topology
defined by the family of {\it all} seminorms. It is easy to see that
every linear subspace of $\phi$ is closed, which easily leads to the
following observation.

\begin{proposition}\label{prop2} The space $\phi$ has the invariant
subspace property.
\end{proposition}

Contrary to the common perception, there is an abundance of complete
topological vector spaces of countable dimension. The main result of
this paper is the following theorem.

\begin{theorem}\label{main} There is a complete locally convex
topological vector space $X$ of countable algebraic dimension such
that $X$ does not have the invariant subspace property.
\end{theorem}

In other words, Theorem~\ref{main} provides a complete locally
convex topological vector space $X$ with $\dim X=\aleph_0$ and $T\in
L(X)$ such that $T$ does not have non-trivial closed invariant
subspaces.

\section{Proof of Propositions~\ref{prop1} and~\ref{prop2}}

Although Propositions~\ref{prop1} and~\ref{prop2} are certainly
known facts, we do not know whether they can be found in the
literature or whether they are folklore. Granted that their proofs
are fairly elementary and short, we present them for the sake of
convenience.

\begin{proof}[Proof of Proposition~$\ref{prop1}$] First, assume that
$X$ is a normed space of countable dimension. Let $B$ be the
completion of $X$. Then $B$ is a separable infinite dimensional
Banach space. According to Ansari \cite{ansa1} and
Bernal--Gonz\'ales \cite{bernal}, there is a hypercyclic $T\in
L(B)$. That is, there is $x\in B$ such that $\{T^nx:n\in\Z_+\}$ is
dense in $B$. Let $Z$ be the linear span of $T^nx$ for $n\in\Z_+$.
Since $Z$ and $X$ are both dense linear subspaces of $B$, according
to Grivaux \cite{gri2}, there is an invertible $S\in L(B)$ such that
$S(Z)=X$. Since $Z$ is invariant for $T$, $X$ is invariant for
$STS^{-1}$. That is, the restriction $A=STS^{-1}\bigr|_{X}$ belongs
to $L(X)$. Let $u\in X\setminus\{0\}$. Then $S^{-1}u$ is a non-zero
vector in $Z$ and therefore $S^{-1}u=p(T)x$, where $p$ is a non-zero
polynomial. Due to Bourdon \cite{bourd}, $S^{-1}u=p(T)x$ is also a
hypercyclic and therefore cyclic vector for $T$. By similarity, $u$
is a cyclic vector for $STS^{-1}$ and therefore for $A$. Thus every
non-zero vector in $X$ is cyclic for $A$. That is, $A$ has no
non-trivial closed invariant subspaces.

Next, let $B$ be a separable infinite dimensional invariant
subspace. Then there is a dense linear subspace $X$ of $B$ ($X$ can
even be chosen to be a hyperplane \cite{bonet}) such that every
$T\in L(X)$ has the shape $\lambda I+S$ with $\lambda\in \C$ and
$\dim S(X)<\infty$. Trivially, such a $T$ has a one-dimensional
invariant subspace.
\end{proof}

\begin{proof}[Proof of Proposition~$\ref{prop2}$] Let $T\in L(\phi)$
and $x\in \phi\setminus\{0\}$. Then the linear span $L$ of
$\{T^nx:n\in\Z_+\}$ is an (automatically closed) invariant subspace
of $T$ different from $\{0\}$. If $L\neq \phi$, we are done. If
$L=\phi$, then the vectors $T^nx$ are linearly independent
(otherwise $L$ is finite dimensional). Hence the linear span $L_1$
of $\{T^nx:n\in\N\}$ is a hyperplane in $L$. Clearly $L_1$ is
$T$-invariant and non-trivial.
\end{proof}

\section{Proof of Theorem~\ref{main}}

We shall construct an operator $T$ with no non-trivial invariant
subspaces, needed in order to prove Theorem~\ref{main}, by lifting a
non-linear map on a topological space to a linear map on an
appropriate topological vector space.

\subsection{A class of complete countably dimensional spaces}

Recall that a topological space $X$ is called {\it completely
regular} (or {\it Tychonoff}) if for every $x\in X$ and a closed
subset $F\subset X$ satisfying $x\notin F$, there is a continuous
$f:X\to \R$ such that $f(x)=1$ and $f\bigr|_F=0$. Equivalently, a
topological space is completely regular, if its topology can be
defined by a family of pseudometrics. Note that any subspace of a
completely regular space is completely regular and that every
topological group is completely regular.

Our construction is based upon the concept of the {\it free locally
convex space} \cite{usp}. Let $X$ be a completely regular
topological space. We say that a topological vector space $L_X$ is a
free locally convex space of $X$ if $L_X$ is locally convex,
contains $X$ as a subset with the topology induced from $L_X$ to $X$
being the original topology of $X$ and for every continuous map $f$
from $X$ to a locally convex space $Y$ there is a unique continuous
linear operator $T:L_X\to Y$ such  that $T\bigr|_X=f$. It turns out
that for every completely regular topological space $X$, there is a
free locally convex space $L_X$ unique up to an isomorphism leaving
points of $X$ invariant. Thus we can speak of {\it the} free locally
convex space $L_X$ of $X$. Note that $X$ is always a Hamel basis in
$L_X$. Thus, as a vector space, $L_X$ consists of formal finite
linear combinations of elements of $X$. Identifying $x\in X$ with
the point mass measure $\delta_x$ on $X$ ($\delta_x(A)=1$ if $x\in
A$ and $\delta_x(A)=0$ if $x\notin A$), we can also think of
elements of $L_X$ as measures with finite support on the
$\sigma$-algebra of all subsets of $X$. Under this interpretation
$$
L_X^0=\{\mu\in L_X:\mu(X)=0\}
$$
is a closed hyperplane in the locally convex space $L_X$. If $f:X\to
X$ is a continuous map, from the definition of the free locally
convex space it follows that $f$ extends uniquely to a continuous
linear operator $T_f\in L(L_X)$. It is also clear that $L_X^0$ is
invariant for $T_f$. Thus the restriction $S_f$ of $T_f$ to $L_X^0$
belongs to $L(L_X^0)$.

According to Uspenskii \cite{usp}, $L_X$ is complete if and only if
$X$ is Dieudonne complete and every compact subset of $X$ is finite.
Since Dieudonne completeness follows from paracompactness, every
regular countable topological space is Dieudonne complete. Since
every countable compact topological space is metrizable, for a
countable $X$, finiteness of compact subsets is equivalent to the
absence of non-trivial convergent sequences (a convergent sequence
is trivial if it is eventually stabilizing). Note also that a
regular countable topological space is automatically completely
regular and therefore we can safely replace the term 'completely
regular' by 'regular' in the context of countable spaces. Thus we
can formulate the following corollary of the Uspenskii theorem.

\begin{proposition}\label{prop3} Let $X$ be a regular countable
topological space. Then the countably dimensional locally convex
topological vector spaces $L_X$ and $L_X^0$ are complete if and only
if there are no non-trivial convergent sequences in $X$.
\end{proposition}

The above proposition provides plenty of complete locally convex
spaces of countable algebraic dimension. We also need the shape of
the dual space of $L_X$. As shown in \cite{usp}, $L_X'$ can be
identified with the space $C(X)$ of continuous scalar valued
functions on $X$ in the following way. Every $f\in C(X)$ produces a
continuous linear functional on $L_X$ in the usual way:
$$
\langle f,\mu\rangle=\int f\,d\mu=\sum c_jf(x_j),\ \ \text{where}\ \
\mu=\sum c_j\delta_{x_j}
$$
and there are no other continuous linear functionals on $L_X$.

\subsection{Operators $S_f$ with no invariant subspaces}

The following lemma is the main tool in the proof of
Theorem~\ref{main}.

\begin{lemma}\label{m1}Let $\tau$ be a regular topology on $\Z$ such that
$f:\Z\to \Z$, $f(n)=n+1$ is a homeomorphism of $\Z_\tau=(\Z,\tau)$
onto itself, $\Z_+$ is dense in $\Z_\tau$ and for every
$z\in\C\setminus\{0,1\}$, $n\mapsto z^n$ is non-continuous as a map
from $\Z_\tau$ to $\C$. Then the operators $T_f$ and $S_f$ are
invertible continuous linear operators on $L_{\Z_\tau}$ and
$L^0_{\Z_\tau}$ respectively and $S_f$ has no non-trivial closed
invariant subspaces.
\end{lemma}

\begin{proof} We already know that $T_f$ and $S_f$ are continuous
linear operators. It is easy to see that $T_f^{-1}=T_{f^{-1}}$ and
$S_f^{-1}=S_{f^{-1}}$. Since $f^{-1}$ is also continuous, $T_f$ and
$S_f$ have continuous inverses.

Now let $\mu\in L_{\Z_\tau}\setminus\{0\}$. It remains to show that
$\mu$ is a cyclic vector for $S_f$. Assume the contrary. Then there
is a non-constant $g\in C(\Z_\tau)$ such that $\langle
S_f^n\mu,g\rangle=\langle T_f^n\mu,g\rangle=0$ for every $n\in\Z_+$.
Decomposing $\mu$ as a linear combination of point mass measures, we
have $\mu=\sum\limits_{k=-l}^l c_k\delta_k$ with $c_k\in\C$. Then
$\mu=\sum\limits_{j=0}^{2l}c_{j-k}T_f^{j}\delta_{-l}=p(T_f)\delta_{-l}$,
where $p$ is a non-zero polynomial. Then $0=\langle
T_f^n\mu,g\rangle=\langle T_f^n\delta_{-l},p(T_f)'g\rangle$ for
$n\in\Z_+$. Thus the functional $p(T_f)'g$ vanishes on the linear
span of $T_f^n\delta_{-l}$ with $n\in\Z_+$, which contains the
linear span of $\Z_+$ in $L_{\Z_\tau}$. Since $\Z_+$ is dense in
$\Z_\tau$, $p(T_f)'g$ vanishes on a dense linear subspace and
therefore $p(T_f)'g=0$. It immediately follows that $T_f'$ has an
eigenvector, which is given by a non-constant function $h\in C(X)$:
$T_f'h=zh$ for some $z\in\C$. Since $T_f$ is invertible, so is
$T_f'$ and therefore $z\neq 0$. It is easy to see that
$T_f'h(n)=h(n+1)$ for each $n\in\Z$. Thus the equality $T_f'h=zh$
implies that (up to a multiplication by a non-zero constant)
$h(n)=z^n$ for each $n\in\Z$. Since $h$ is non-constant, $z\neq 1$.
Thus the map $n\mapsto z^n$ is continuous on $\Z_\tau$ for some
$z\in\C\setminus\{0,1\}$. We have arrive to a contradiction.
\end{proof}

\subsection{A specific countable topological space}

\begin{lemma}\label{m2} There exists a regular topology $\tau$ on $\Z$ such
that the topological space $\Z_\tau=(\Z,\tau)$ has the following
properties
\begin{itemize}\itemsep-2pt
\item[\rm (a)]$f:\Z\to \Z$, $f(n)=n+1$ is a homeomorphism of $\Z_\tau$ onto
itself$\,;$
\item[\rm (b)]$\Z_+$ is dense in $\Z_\tau;$
\item[\rm (c)]for every $z\in\C\setminus\{0,1\}$,
$n\mapsto z^n$ is non-continuous as a map from $\Z_\tau$ to $\C;$
\item[\rm (d)]$\Z_\tau$ has no non-trivial convergent sequences.
\end{itemize}
\end{lemma}

\begin{proof}Consider the Hilbert space $\ell_2(\Z)$ and the
bilateral weighted shift $T\in L(\ell_2(\Z))$ given by
$Te_n=e_{n-1}$ if $n\leq 0$ and $Te_n=2e_{n-1}$ if $n>0$, where
$\{e_n\}_{n\in\Z}$ is the canonical orthonormal basis in
$\ell_2(\Z)$. Symbol $\hh_\sigma$ stands for $\ell_2(\Z)$ equipped
with its weak topology $\sigma$. Clearly $T$ is invertible and
therefore $T$ is a homeomorphism on $\hh_\sigma$. According to Chan
and Sanders, there is $x\in\ell_2(\Z)$ such that the set
$O=\{T^nx:n\in\Z_+\}$ is dense in $\hh_\sigma$. Then
$Y=\{T^nx:n\in\Z\}$ is also dense in $\hh_\sigma$. We equip $Y$ with
the topology inherited from $\hh_\sigma$ and transfer it to $\Z$ by
declaring the bijection $n\mapsto T^nx$ from $\Z$ to $Y$ a
homeomorphism.

Since $\sigma$ is a completely regular topology, so is the just
defined topology $\tau$ on $\Z$. Since $T$ is a homeomorphism on
$(\ell_2(\Z),\sigma)$ and $Y$ is a subspace of the topological space
$(\ell_2(\Z),\sigma)$ invariant for both $T$ and $T^{-1}$, $T$ is a
homeomorphism on $Y$. Since $T(T^nx)=T^{n+1}x$, it follows that $f$
 is a homeomorphism on $\Z_\tau$. Density of $O$ in $Y$ implies the density
of $\Z_+$ in $\Z_\tau$.

Observe that the sequence $\{\|T^nx\|\}_{n\in\Z}$ is strictly
increasing and $\|T^nx\|\to\infty$ as $n\to+\infty$. Indeed, the
inequality $\|Tu\|\geq \|u\|$ for $u\in\ell_2(\Z)$ follows from the
definition of $T$. Hence $\{\|T^nx\|\}_{n\in\Z}$ is increasing.
Assume that $\|T^{n+1}x\|=\|T^nx\|$ for some $n\in\Z$. Then, by
definition of $T$, $T^nx$ belongs to the closed linear span $L$ of
$e_n$ with $n<0$. The latter is invariant for $T$ and therefore
$T^mx\in L$ for $m\geq n$, which is incompatible with the
$\sigma$-density of $O$. Next, if $\|T^nx\|$ does not tend to
$\infty$ as $n\to+\infty$, the sequence $\{\|T^nx\|\}_{n\in\Z_+}$ is
bounded. Since every bounded subset of $\ell_2(\Z)$ is
$\sigma$-nowhere dense, we have again obtained a contradiction with
the $\sigma$-density of $O$.

In order to show that $X$ has no non-trivial convergent sequences,
it suffices to show that $Y$ has no non-trivial convergent
sequences. Assume that $\{T^{n_k}x\}_{k\in\Z_+}$ is a non-trivial
convergent sequence in $Y$. Without loss of generality, we can
assume that the sequence $\{n_k\}$ of integers is either strictly
increasing or strictly decreasing. If $\{n_k\}$ is strictly
increasing the above observation ensures that
$\|T^{n_k}x\|\to\infty$ as $k\to\infty$. Since every
$\sigma$-convergent sequence is bounded, we have arrived to a
contradiction. If $\{n_k\}$ is strictly decreasing, then by the
above observation, the sequence $\{\|T^{n_k}x\|\}$ of positive
numbers is also strictly decreasing and therefore converges to
$c\geq 0$. Then $\|T^lx\|>c$ for every $l\in\Z$. Since
$\{T^{n_k}x\}$ $\sigma$-converges to $T^mx\in Y$, the upper
semicontinuity of the norm function with respect to $\sigma$ implies
that $\|T^mx\|\leq c$ and we have arrived to a contradiction.

Finally, let $z\in\C\setminus\{0,1\}$ and $f:\Z\to \C$, $f(n)=z^n$.
It remains to show that $f$ is not continuous as a function on
$\Z_\tau$. Equivalently, it is enough to show that the function
$g:Y\to\C$, $g(T^nx)=z^n$ is non-continuous. Assume the contrary.
There are two possibilities. First, consider the case $|z|\neq 1$.
In this case the the topology on the set $M=\{z^n:n\in\Z\}$
inherited from $\C$ is the discrete topology. Continuity of the
bijection $g:Y\to M$ implies then that $Y$ is also discrete, which
is apparently not the case: the density of $Y$ in $\hh_\sigma$
ensures that $Y$ has no isolated points. It remains to consider the
case $|z|=1$, $z\neq1$. In this case the closure $G$ of $\{z^n:n\in
\Z\}$ is a closed subgroup of the compact abelian topological group
$\T$. Since $x$ is a hypercyclic vector for $T$ acting on
$\hh_\sigma$ and $z$ generates the compact abelian topological group
$G$, \cite[Corollary~4.1]{66} implies that $\{(T^nx,z^n):n\in\Z_+\}$
is dense in $\hh_\sigma\times G$. It follows that the graph of $g$
is dense in $Y\times G$. Since $G$ is not a single-point space, the
latter is incompatible with the continuity of $g$. The proof is now
complete.
\end{proof}

\begin{proof}[Proof of Theorem~$\ref{main}$] Let $\tau$ be
the topology on $\Z$ provided by Lemma~\ref{m2}. By
Proposition~\ref{prop3}, $E=L^0_{\Z_\tau}$ is a complete locally
convex space of countable algebraic dimension. Let $f:\Z\to\Z$,
$f(n)=n+1$ and $S=S_f\in L(E)$. By Lemmas~\ref{m1} and~\ref{m2}, $S$
is invertible and has no non-trivial invariant subspaces. The proof
is complete (with an added bonus of invertibility of $S$).
\end{proof}

It is easy to see that the topology $\tau$ on $\Z$ constructed in
the proof of Lemma~\ref{m2} does not agree with the group structure.
That is $\Z_\tau$ is not a topological group. Indeed, it is easy to
see that the group operation $+$ is only separately continuous on
$\Z_\tau$, but not jointly continuous. As a matter of curiosity, it
would be interesting to find out whether there exists a topology
$\tau$ on $\Z$ satisfying all conditions of Lemma~\ref{m2} and
turning $\Z$ into a topological group.

\small\rm

\vskip1truecm

\scshape

\noindent Stanislav Shkarin

\noindent Queens's University Belfast

\noindent Department of Pure Mathematics

\noindent University road, Belfast, BT7 1NN, UK

\noindent E-mail address: \qquad {\tt s.shkarin@qub.ac.uk}

\end{document}